\documentclass[12pt, reqno, twoside, letterpaper]{amsart}
\usepackage{square-totients}

\title[A note on square totients]
      {A note on square totients}

\author[T.\ Freiberg]{Tristan Freiberg}

\address{Department of Mathematics, 
         University of Missouri, 
         Columbia MO, USA.}

\email{freibergt@missouri.edu}

\author[C.\ Pomerance]{Carl Pomerance}

\address{Department of Mathematics, 
         Dartmouth College, 
         Hanover NH, USA.}

\email{carl.pomerance@dartmouth.edu}

\date{\today}

\begin{document}


\begin{abstract}
A well-known conjecture asserts that there are infinitely many 
primes $p$ for which $p - 1$ is a perfect square.
We obtain upper and lower bounds of matching order on the 
number of pairs of distinct primes $p,q \le x$ for which 
$(p - 1)(q - 1)$ is a perfect square.
\end{abstract}

\maketitle


\section{Introduction}
 \label{sec:1}

The first of ``Landau's problems'' on primes is to show that there 
are infinitely many primes $p$ for which $p - 1=\oldsquare$, that 
is, a perfect square.
Heuristics \cite{BH, SHA} suggest that 
\[
 \#\{p \le x : p - 1 = \oldsquare\} 
  \sim 
  {\textstyle \frac{1}{2} }
    \mathfrak{S} \int_2^{\sqrt{x}} \frac{\dd{t}}{\log t}
     \qquad (x \to \infty), 
\]
where 
$
 \mathfrak{S}
  \defeq
   \prod_{p > 2}\big(1 - (-1/p)/(p - 1)\big)
$
and $(-1/\cdot)$ is the Legendre symbol.
The problem being as unassailable now as it was in 1912 when 
Landau compiled his famous list, we consider the problem of 
counting pairs $(p,q)$ of distinct primes for which 
$(p - 1)(q - 1) = \oldsquare$. 

Let $\PP$ denote the set of all primes and let 
\[
  \SS  
  \defeq 
    \Br{
        (p,q) \in \PP \times \PP : p \ne q \,\, \hbox{and} \,\, 
        (p - 1)(q - 1) = \oldsquare  
       }.       
\]
For $x \ge 2$, let 
\[
  \SS(x) \defeq \#\{(p,q) \in \SS : p,q \le x\},
\]

\begin{theorem}
 \label{thm:1}
There exist absolute constants $c_2 > c_1 > 0$ such that for all 
$x \ge 5$, 
\[
 c_1x/\log x < \SS(x) < c_2x/\log x.
\]
\end{theorem}

We remark that the lower bound $\SS(x) \gg x/\log x$ gives 
\[
 \SS'(x)
  \defeq 
   \#\{n \le x : n = pq, (p,q) \in \SS\}
    \ge 
    {\textstyle \frac{1}{2} }
      \SS(\sqrt{x})
       \gg 
        \sqrt{x}/\log x,
\]
improving on the bound 
$
 \SS'(x) \gg \sqrt{x}/(\log x)^4 
$
of the first author \cite[Theorem 1.2]{FRE}, and independently, 
\cite{BL}.  
Let $\phi$ denote Euler's function.  
Note that for primes $p,q$ we have $\phi(pq) = \oldsquare$ if and 
only if $(p,q) \in \SS$.  
The distribution of integers $n$ with $\phi(n) = \oldsquare$ has 
been considered recently also in \cite{BFPS} and
\cite[Section 4.8]{DKL}, while the distribution of 
integers $n$ with $n^2$ a totient (that is, a value of $\phi$) has 
been considered in \cite{PP}.  
We remark that our proof goes over with trivial 
modifications to the case of $(p + 1)(q + 1) = \oldsquare$, that 
is, $\sigma(pq) = \oldsquare$, where $\sigma$ is the 
sum-of-divisors function.
A similar result is to be expected for solutions to 
$(p + b)(q + b) = \oldsquare$ for any fixed nonzero integer $b$.

In \cite{BL, FRE} solutions to $(p - 1)(q - 1)(r - 1) = m^3$ are 
also considered, where $p,q,r$ are distinct primes, and more 
generally $\phi(n) = m^k$, where $n$ is the product of $k$ 
distinct primes.
In \cite{BL}, the authors show that if the primes in $n$ are 
bounded by $x$, there are at least $c_kx/(\log x)^{2k}$ solutions, 
while in \cite{FRE}, it is shown that there are at least 
$c_kx/(\log x)^{k + 2}$ solutions.
Our lower bound construction in the present paper can be 
extended to give at least $c_kx/(\log x)^{k - 1}$ solutions.
We do not have a matching upper bound when $k \ge 3$.

In addition to notation already introduced, $p,q$ will always
denote primes, 
$\ind{\PP}$ denotes the indicator function of $\PP$,
\[
 \pi(x) \defeq \sum_{p \le x} 1, 
  \quad 
   \pi(x;k,b) \defeq \sums[p \le x][p \equiv b \bmod k] 1,
\]
$\Lambda$ denotes the von Mangoldt function, 
$\mu$ denotes the M\"obius function,
$\omega(n)$ denotes the number of distinct prime divisors of $n$, 
and $(D/\cdot)$ denotes the Legendre/Kronecker symbol.
Note that $A = \Oh[B]$, $A \ll B$ and $B \gg A$ all indicate that 
$|A| \le c|B|$ for some absolute constant $c$, $A \asymp B$ means 
$A \ll B \ll A$, $A = \Oh[B][\alpha]$ and $A \ll_{\alpha} B$ 
denote that $|A| \le c(\alpha)|B|$ for some constant $c$ depending 
on $\alpha$, and $A \asymp_{\alpha} B$ denotes that 
$A \ll_{\alpha} B \ll_{\alpha} A$.
Also, $A = \oh[B]$ indicates that $|A| \le c(x)|B|$ for some 
function $c(x)$ of $x$ that goes to zero as $x$ tends to infinity.


\section{Auxiliary lemmas}
 \label{sec:2}

We will use the following bounds in the proof of 
Theorem \ref{thm:1}. 

\begin{lemma}
 \label{lem:2.1}
\textup{(}i\textup{)} If $x \ge 2$ and $d \ge 1$ then
\[
 \sum_{n \le x} \frac{1}{\phi(n)} \ll \log x,
\quad \sum_{n>x}\frac1{\phi(n^2)}\asymp\frac1x,
  \quad 
   \text{and}
    \quad 
     \sums[n > x][d \mid n^2] \frac{1}{\phi(n^2)}
      \ll 
       \frac{d^{1/2}}{\phi(d)x}.
\]
\textup{(}ii\textup{)}
If $n \ge 2$ then 
\[
 \sum_{m < n}
  \frac{n^2 - m^2}{\phi(n^2 - m^2)}
   \ll
    n.
\]
\end{lemma}

\begin{proof}
(i) We have 
$
 \sum_{n \le x} 1/n \le 1 + \int_1^n \dd{t}/t = 1 + \log x,
$
and the first bound follows by using the identity
$
 n/\phi(n) = \sum_{m \mid n} \mu(m)^2/\phi(m)
$
and switching the order of summation.
The second bound follows similarly, noting that
$\sum_{n > x^2} 1/n^2 \asymp 1/x$ and that 
$\phi(n^2) = n\phi(n)$.
For the third bound, write $d = d_1d_2^2$, where $d_1$ is
squarefree, and note that $d\mid n^2$ if and only if 
$d_1d_2 \mid n$.
Thus, 
\begin{equation}
 \label{eq:2.1}
  \sums[n > x][d\mid n^2]
   \frac{1}{\phi(n^2)}
  = \sums[n > x][d_1d_2 \mid n]
     \frac{1}{n\phi(n)}
      \le
       \frac1{d_1d_2\phi(d_1d_2)}
        \sum_{m > x/(d_1d_2)}
         \frac1{\phi(m^2)}.
\end{equation}
If $d_1d_2 \le x/2$, this last sum is, by the second part, 
$\Oh[d_1d_2/x]$, leading to 
\[
 \sums[n > x][d \mid n^2]
  \frac{1}{\phi(n^2)}
   \ll 
    \frac{1}{\phi(d_1d_2)x}
   = \frac{d}{\phi(d)d_1d_2x}
      \le
       \frac{d^{1/2}}{\phi(d)x}.
\]
Finally, if $d_1d_2 > x/2$, the last sum in \eqref{eq:2.1} is 
$\Oh[1]$, leading to 
\[
 \sums[n > x][d \mid n^2] 
  \frac{1}{\phi(n^2)}
   \ll 
    \frac{1}{d_1d_2\phi(d_1d_2)}
     \ll 
      \frac{1}{x\phi(d_1d_2)}
       \le
        \frac{d^{1/2}}{\phi(d)x}.
\]

(ii) For any positive integer $k$ we have
\[
 \frac{k}{\phi(k)}
  = 
   \sums[d \mid k][d^2 \le k]
    \frac{\mu(d)^2}{\phi(d)}
     + 
      \sums[d \mid k][d^2 > k]
       \frac{\mu(d)^2}{\phi(d)}
  =
   \sums[d \mid k][d^2 \le k]
    \frac{\mu(d)^2}{\phi(d)}
   + \Oh[k^{-1/3}]
  \ll
   \sums[d \mid k][d^2 \le k]
    \frac{\mu(d)^2}{\phi(d)},
\]
using the elementary bounds 
\[
 d/\phi(d) \ll \log\log(3d) 
 \quad \text{and} \quad 
 \textstyle
  \sum_{d \mid k} \mu(d)^2 
  = 2^{\omega(k)} 
  = k^{\Oh(1/\log\log k)}.
\]
Thus,
\[
 \sum_{m < n} 
  \frac{n^2 - m^2}{\phi(n^2 - m^2)}
   \ll
    \sum_{m < n} 
     \hspace{3pt}
      \sums[d \mid n^2 - m^2][d < n]
       \frac{\mu(d)^2}{\phi(d)}
        =
         \sum_{d < n}
          \frac{\mu(d)^2}{\phi(d)}
           \sums[m < n][d \mid n^2 - m^2]
            1.
\]
If $d$ is squarefree and $d \mid n^2 - m^2$, then 
$d = d_1d_2$ for some $d_1,d_2$ with 
$n + m \equiv 0 \bmod d_1$ and
$n - m \equiv 0 \bmod d_2$.
These congruences are satisfied by a unique $m$ modulo 
$d_1d_2 = d$, and there are $2^{\omega(d)}$ ways of writing a 
squarefree integer $d$ as an ordered product of $2$ positive 
integers.
Hence
\[
 \sum_{d < n}
  \frac{\mu(d)^2}{\phi(d)}
   \sums[m < n][d \mid n^2 - m^2] 1
    = 
     \sum_{d < n} 
      \frac{\mu(d)^2}{\phi(d)}
       \sum_{d_1d_2 = d}
        \sums[m < n]
             [d_1 \mid n + m]
             [d_2 \mid n - m] 
             1  
        \ll 
         n\sum_{d < n}
           \frac{\mu(d)^2 2^{\omega(d)}}{d\phi(d)} 
            \ll n.
\]
\end{proof}

We will need uniform bounds for $\pi(x;k,b)$ for $k$ up to a small 
power of $x$. 
The following form of the Brun--Titchmarsh inequality is a 
consequence of a sharp form of the large sieve inequality due to 
Montgomery and Vaughan \cite{MV}.

\begin{lemma} 
 \label{lem:2.2}
If $1 \le k < x$ and $(b,k) = 1$ then
\[
 \pi(x;k,b) < \frac{2x}{\phi(k)\log(x/k)}.
\]
\end{lemma}

\begin{proof}
See \cite[Theorem 2]{MV}.
\end{proof}

We do not have a matching lower bound for all $k$ up to a power 
of $x$ because of putative Siegel zeros, however these only 
affect a very few moduli $k$ that are multiples of certain 
``exceptional'' moduli.

\begin{lemma}
 \label{lem:2.3}
For any given $\epsilon, \delta > 0$, there exist numbers 
$\eta_{\epsilon,\delta} > 0$, $x_{\epsilon,\delta}$, 
$D_{\epsilon,\delta}$ such that whenever 
$x \ge x_{\epsilon,\delta}$, there is a set 
$\mathcal{D}_{\epsilon,\delta}(x)$, of at most 
$D_{\epsilon,\delta}$ integers, for which 
\[
  \bigg|\pi(x;k,b) - \frac{x}{\phi(k)\log x}\bigg|
   \le 
    \frac{\epsilon x}{\phi(k)\log x}
\]
whenever $k$ is not a multiple of any element of 
$\mathcal{D}_{\epsilon,\delta}(x)$, $k$ is in the range 
\[
 1 \le k \le x^{-\delta + 5/12},
\] 
and $(b,k) = 1$.
Furthermore, every integer in $\mathcal{D}_{\epsilon,\delta}(x)$ 
exceeds $\log x$, and all, but at most one, exceed 
$x^{\eta_{\epsilon,\delta}}$.
\end{lemma}

\begin{proof}
See \cite[Theorem 2.1]{AGP}.
\end{proof}

In fact we will need to count primes $p \equiv b \bmod k$ for 
which the quotient $(p - b)/k$ is squarefree.
We apply an inclusion-exclusion argument to Lemma \ref{lem:2.3}.

\begin{lemma}
 \label{lem:2.4}
There exist absolute constants $\eta > 0$, $x_0$, 
$D$ such that whenever $x \ge x_0$, there is a set 
$\mathcal{D}(x)$, of at most $D$ integers, for which 
\[
  \sum_{a \le x/k} \mu(a)^2\ind{\PP}(ak + b)
   >
    \frac{x}{100\phi(k)\log x}
\]
whenever $36k$ is not a multiple of any element of 
$\mathcal{D}(x)$, $k$ is in the range $1 \le k \le x^{1/3}$,
and $(b,k) = 1$ with $1 \le b < k$.
Furthermore, every integer in $\mathcal{D}(x)$ exceeds $\log x$, 
and all, but at most one, exceed $x^{\eta}$.
\end{lemma}

\begin{proof}
Let $1 \le b < k \le x^{1/3}$ with $(b,k) = 1$.
Using  
$\mu(a)^2 \ge 1 - \sum_{p^2 \mid a} 1$ 
and switching the order of summation, we obtain 
\begin{align*}
 \sum_{a \le x/k}
  \mu(a)^2
   \ind{\PP}(ak + b)
  & \ge
     \sum_{a \le x/k} 
      \ind{\PP}(ak + b)
     - \sum_{p \le \sqrt{x/k}}
        \sum_{c \le x/(p^2k)} 
         \ind{\PP}(cp^2k + b)
 \\
  & \ge 
     \pi(x;k,b)
    - \sum_{p \le \sqrt{x/k}}
       \pi(x;p^2k,b)
      - \sqrt{x/k}.
\end{align*}

Let $1 \le y < z < \sqrt{x/k}$.
Trivially, we have
\[
 \sum_{z < p \le \sqrt{x/k}}
  \pi(x;p^2k,b)
   \le 
    \sum_{p > z}
     \frac{x}{p^2k}
      \ll 
       \frac{x}{kz\log z}.
\]
Here we have used the bound
$\sum_{p > z} 1/p^2 \ll 1/(z\log z)$, which follows from 
the bound $\pi(x) \ll x/\log x$ by partial summation.
By Lemma \ref{lem:2.2} we have
\[
 \sum_{y < p \le z}
  \pi(x;p^2k,b)
   < 
    \frac{2x}{\log(x/(z^2k))}
     \sum_{p > y}
      \frac{1}{\phi(p^2k)}
       \le 
        \frac{2x}{\phi(k)\log(x/(z^2k))}\sum_{p>y}\frac1{p(p-1)},
\]
using $\phi(p^2k)\ge\phi(p^2)\phi(k)$.

We set $y = 3$ and $z = \log x$ so that 
$\log(x/(z^2k)) \sim \log(x/k) \ge \frac{2}{3}\log x$.
We verify that $\sum_{p > 3}1/(p(p-1)) < 0.1065$.
Combining everything gives 
\[
 \sum_{a \le x/k}
  \mu(a)^2
   \ind{\PP}(ak + b)
    >
     \pi(x;k,b) - \pi(x;4k,b) - \pi(x;9k,b)
    - \frac{0.32x}{\phi(k)\log x}
\]
for all sufficiently large $x$.
We complete the proof by applying Lemma \ref{lem:2.3} with 
$\epsilon = 1/1000$ and $\delta = 1/12$, noting that
$1 - 1/2 - 1/6 - 3\epsilon - 0.32 > 1/100$.
\end{proof}

We remark that with more work, a version of Lemma \ref{lem:2.4} 
can be proved as an equality, with the factor $1/100$ replaced 
with $c_k+o(1)$ (as $x \to \infty$), where $c_k$ is 
Artin's constant $\prod_p(1-1/(p(p - 1))$ times 
$\prod_{p \mid k}(1 - 1/(p^3 - p^2 - p))$.

\begin{lemma}
 \label{lem:2.5}
Fix $\delta \in (0,1]$ and let $x \ge 3$. 
There is a set $\mathcal{E}_{\delta}(x)$ of quadratic, primitive 
characters, all of conductor less than $x$, satisfying 
$\#\mathcal{E}_{\delta}(x) \ll_{\delta} x^{\delta}$ and such that 
the following holds.
If $\chi$ is a real, primitive character of conductor $d \le x$ 
and $\chi \not\in \mathcal{E}_{\delta}(x)$, then 
\[
 \prod_{y < p \le z}
  \br{1 - \frac{\chi(p)}{p}}
   \asymp_{\delta}
    1
\]
uniformly for $z > y \ge \log x$.
\end{lemma}

\begin{proof}
See \cite[Lemma 3.3]{CDKS}.
The authors of \cite{CDKS} state that the proof of their lemma 
borrows from \cite[Proposition 2.2]{GS}, and the authors of 
\cite{GS} state that their proposition is essentially due to 
Elliott \cite{E}.
(The lemma, as stated here, is quoted from \cite[Lemma 7]{PP}, 
and is equivalent to \cite[Lemma 3.3]{CDKS}.)
\end{proof}

\begin{lemma}
 \label{lem:2.6}
If $x \ge 2$ then 
\[
 \sum_{a \le x} 
  \frac{a\mu(a)^2}{\phi(a)^2}
   \prod_{2 < p \le \sqrt{x}}
    \br{1 - \frac{(-a/p)}{p}}^2
     \ll
      \log x.
\]
\end{lemma}

\begin{proof}
First, we note that for $y \ge 1$ we have the elementary bound
\begin{equation}
 \label{eq:2.2}
  \sum_{a > y}
   \frac{a^2}{\phi(a)^4}
    \ll
     \frac{1}{y}.
\end{equation}
To see this, let $h$ be the multiplicative function satisfying
$a^4/\phi(a)^4 = \sum_{m \mid a} h(m)$, so that 
\[
 h(m) = \mu(m)^2 \prod_{p \mid a} \br{\frac{p^4}{p^4 - 1} - 1}.
\]
Then 
\begin{align*}
 \sum_{a > y}
  \frac{a^2}{\phi(a)^4}
  & =
    \sum_{a > y}
     \frac{1}{a^2} 
      \frac{a^4}{\phi(a)^4}
    =
    \int_{y}^{\infty}
     \frac{2}{t^3}
      \sum_{y < a \le t}
       \frac{a^4}{\phi(a)^4} \dd t
  \\
  & \le 
     \int_{y}^{\infty}
      \frac{2}{t^2}
       \sum_{m \le t}
        \frac{h(m)}m \dd t
       < \frac{2}{y}
          \sum_{m \ge 1}
           \frac{h(m)}m.
\end{align*}
This last sum has a convergent Euler product, so \eqref{eq:2.2} is 
established.

For a positive squarefree integer $a$, let $\chi_a$ be the 
Dirichlet character that sends an odd prime $p$ to $(-a/p)$, and 
such that $\chi_a(2) = 1$ or $0$ depending on whether 
$a \equiv 3 \bmod 4$ or not, respectively.  
The character $\chi_a$ is primitive and has conductor $a$ if 
$a \equiv 3 \bmod 4$ and $4a$ otherwise.

%
The product in the lemma (without being squared) resembles 
$L(1,\chi_a)^{-1}$, in fact,
\[
 L(1,\chi_a)^{-1} = \prod_p \br{1-\frac{(-a/p)}p}.
\]
Our first goal is to show that we uniformly have 
\begin{equation}
 \label{eq:2.3}
 L(1,\chi_a) 
  \prod_{2 < p \le \sqrt{x}} 
   \br{1-\frac{(-a/p)}p}
    \ll 1
\end{equation}
for all small $a$ and most other values of $a \le x$.
Suppose that $a \le (\log x)^4$.  
Considering the $\phi(4a)$ residue classes $r$ mod $4a$ that are 
coprime to $4a$, we see (since the conductor of $\chi_a$ divides 
$4a$) that $(-a/p) = 1$ or $-1$ depending on which class $p$ lies 
in, with $\frac{1}{2}\phi(4a)$ classes giving $1$ and 
$\frac{1}{2}\phi(4a)$ classes giving $-1$.
It follows from the Siegel--Walfisz theorem \cite[\S22 (4)]{DAV} 
that
\[
 \sum_{p > \sqrt{x}}
  \frac{(-a/p)} p
 = \int_{\sqrt{x}}^{\infty}
    \frac{1}{t^2}
     \sum_{\sqrt{x} < p \le t}
      (-a/p) \dd t
       \ll
        \phi(4a)
         \int_{\sqrt{x}}^{\infty}
          \frac{1}{t(\log t)^5} \dd t
           \ll 1.
\]
Exponentiating, we get \eqref{eq:2.3}.

Now suppose that $a > (\log x)^4$.  
We break the interval $((\log x)^4,x]$ into dyadic intervals of
the form $I_j \defeq [2^j,2^{j+1})$, where the first and last 
intervals may overshoot a bit.  
Using Lemma \ref{lem:2.5} with $\delta = \frac{1}{4}$, 
$y = \sqrt{x}$, and letting $z \to \infty$, we have \eqref{eq:2.3} 
for all $a \in I_j$ except for possibly $\Oh[2^{j/4}]$ values of 
$a$.
Using the trivial estimate 
\[
 \prod_{2 < p \le \sqrt{x}}
  \br{1 - \frac{(-a/p)}p}^2
   \ll 
   (\log x)^2
\]
and $a/\phi(a)^2 \ll (\log\log a)^2a^{-1}$, the contribution of 
these exceptional values of $a \in I_j$ to the sum in the lemma is
\[
 \ll
  2^{j/4}(\log j)^22^{-j}(\log x)^2,
\]
which when summed over integers $j$ being considered gives 
$\Oh[(\log\log x)^2/\log x]$.
Thus, we may ignore these exceptional values of $a$, so assuming
that \eqref{eq:2.3} always holds.

By the Cauchy--Schwarz inequality we have
\[
 \sum_{a \in I_j} 
  \frac{a \mu(a)^2}{\phi(a)^2}
   L(1,\chi_a)^{-2}
    \le
     \bigg(
      \sum_{a \in I_j}
       \frac{a^2}{\phi(a)^4}
     \bigg)^{1/2}
     \bigg(
      \sum_{a \in I_j}
       \mu(a)^2
        L(1,\chi_a)^{-4}
     \bigg)^{1/2}.
\]
Now the first sum is $\Oh[2^{-j}]$ by \eqref{eq:2.2}, and the 
second sum is $\Oh[2^j]$ by \cite[Theorem 2]{GS} (with $z = -4$) 
and the subsequent comment about Siegel's theorem.  
Thus, the contribution from $a\in I_j$ to the sum in the lemma is 
$\Oh[1]$, and since there are $\Oh[\log x]$ choices for $j$, the 
lemma is proved.
\end{proof}

We remark that \cite[Section 10]{BZ} has a similar calculation as 
in Lemma \ref{lem:2.6}.


\section{Proof of Theorem \ref{thm:1}}
 \label{sec:3}

Our proof begins with the observation that 
every positive integer has a unique representation of the form 
$an^2$, where $a$ and $n$ are positive integers with $a$ 
squarefree.
Thus, $(p - 1)(q - 1) = \oldsquare$ if and only if 
$p = am^2 + 1$ and $q = an^2 + 1$ for some squarefree $a$.
It follows that for all $x \ge 0$, 
\begin{align}
 \label{eq:3.1}
  \SS(x + 1)
  =
    \sum_{a \le x} 
     \mu(a)^2
      \sums[m, \, n \le \sqrt{x/a}][m \ne n] 
       \ind{\PP}(am^2 + 1)\ind{\PP}(an^2 + 1).
\end{align}

\subsection{The lower bound}
 \label{subsec:3.1}

Let $x \ge 4$ and consider a dyadic interval
\[
 I_y \defeq [y/2,y) \subset[1,x^{1/6}].
\]
Also let 
\begin{equation}
 \label{eq:3.2}
  N_{I_y}(a) \defeq \sum_{n \in I_y} \ind{\PP}(an^2 + 1).
\end{equation}
Letting $\mathscr{I}$ denote a collection of disjoint dyadic 
intervals $I_y$, we deduce from \eqref{eq:3.1} that 
\begin{equation}
 \label{eq:3.3}
  \SS(x + 1)
   \ge 
    \sum_{I_y \in \mathscr{I}}
     \sum_{a \le x/y^2}
      \mu(a)^2
       (N_{I_y}(a)^2 - N_{I_y}(a)).
\end{equation}
By the Cauchy--Schwarz inequality, for every $I_y \in \mathscr{I}$ 
we have
\begin{align}
 \label{eq:3.4}
  \bigg(
   \sum_{a \le x/y^2}
    \mu(a)^2N_{I_y}(a)
  \bigg)^2
   \le 
    \frac{x}{y^2}
     \sum_{a \le x/y^2} \mu(a)^2N_{I_y}(a)^2.
\end{align}

\begin{lemma}
 \label{lem:3.1}
Given an interval $I_y = [y/2,y)$ and an integer $a$, let 
$N_{I_y}(a)$ be as in \eqref{eq:3.2}.
\textup{(}i\textup{)}
Uniformly for $2 \le y < \sqrt{x}$, we have 
\[
  \sum_{a \le x/y^2} 
   N_{I_y}(a)
    \ll 
     \frac{x}{y\log(x/y^2)}.
\]
\textup{(}ii\textup{)}
Uniformly for $2 \le y \le x^{1/6}$, we have 
\[
  \sum_{a \le x/y^2} 
   \mu(a)^2 N_{I_y}(a)
    \gg 
     \frac{x}{y\log x}.
\]
\end{lemma}

\begin{proof}
(i) We change the order of summation and apply Lemma 
\ref{lem:2.2}:
\[
 \sum_{a \le x/y^2} N_{I_y}(a)
  = \sum_{n \in I_y}
   \sum_{a \le x/y^2} 
    \ind{\PP}(an^2 + 1) 
     \ll 
      \sum_{n \in I_y}
       \pi(x; n^2, 1)
        \ll
         \sum_{n \in I_y}
          \frac{x}{\phi(n^2)\log(x/n^2)}.
\]
We have 
$
 \sum_{n \in I_y} 1/\phi(n^2) \ll 1/y
$
by the second bound in Lemma \ref{lem:2.1} (i).

(ii) Let $2 \le y \le x^{1/6}$ and let $I'_y$ be the subset of 
those $n \in I_y$ for which 
\[
 \sum_{a \le x/n^2}
  \mu(a)^2
   \ind{\PP}(an^2 + 1)
    > 
     \frac{x}{100\phi(n^2)\log x}.
\] 
Letting 
$
 N_{I_y'}(a) 
  \defeq 
   \sum_{n \in I_y'} \ind{\PP}(an^2 + 1)
$ 
we see, after switching the order of summation, that
\[
 \sum_{a \le x/y^2} \mu(a)^2 N_{I_y}(a)
  \ge 
   \sum_{a \le x/y^2} \mu(a)^2 N_{I_y'}(a)
    \ge
     \sum_{n \in I_y'} 
      \sum_{a \le x/n^2} 
       \mu(a)^2 
        \ind{\PP}(an^2 + 1),
\]
and hence
\[
 \sum_{a \le x/y^2} \mu(a)^2 N_{I_y}(a)
  > 
   \frac{x}{100\log x}
    \sum_{n \in I_y'}
     \frac{1}{\phi(n^2)}.
\]
We claim that 
\begin{equation}
 \label{eq:3.5}
 \sum_{n \in I_y'} \frac{1}{\phi(n^2)}
  \gg
   \frac{1}{y},
\end{equation}
whence the result.
The claim follows from the second bound in 
Lemma \ref{lem:2.1} (i) if $I_y' = I_y$, so let us 
assume that $I_y' \subsetneq I_y$.

If   
$
 n \in I_y\setminus I'_y
$
then $n^2 \le x^{1/3}$, and so if $x$ is sufficiently large
(as we assume), $36n^2$ is a multiple of an element of the 
``exceptional set'' $\mathcal{D}(x)$ of Lemma \ref{lem:2.4}.
Hence, by the third bound in Lemma \ref{lem:2.1} (i),
\begin{multline*}
 \sum_{n \in I_y\setminus I_y'}
  \frac{1}{\phi(n^2)}
   \le 
    \sum_{d \in \mathcal{D}(x)}
     \sums[n \in I_y][d \mid 36n^2]
      \frac{1}{\phi(n^2)}
       \ll
        \sum_{d \in \mathcal{D}(x)}
         \sums[n \in I_y][d \mid (6n)^2]
          \frac{1}{\phi((6n)^2)}
   \\
     \le 
      \sum_{d \in \mathcal{D}(x)}
       \sums[m \ge 3y][d \mid m]
        \frac{1}{\phi(m^2)}
         \ll
          \frac{1}{y}
           \sum_{d \in \mathcal{D}(x)}
            \frac{d^{1/2}}{\phi(d)}
             \ll
              \frac{\log\log x}{y(\log x)^{1/2}},
\end{multline*}
where the last bound holds because, by Lemma \ref{lem:2.4}, there 
are at most $D$ elements in $\mathcal{D}(x)$, and all elements in 
$\mathcal{D}(x)$ are greater than $\log x$.
Since our estimate is $\oh[1/y]$ as $x \to \infty$, we have 
\eqref{eq:3.5}, and so the lemma.
\end{proof}

\begin{proof}
[Deduction of the lower bound]
Combining \eqref{eq:3.4} with Lemma \ref{lem:3.1} (i) and (ii), 
we see that if $I_y = [y/2,y)$, then, uniformly for 
$(\log x)^2 \le y \le x^{1/6}$,
\begin{align*}
  \sum_{a \le x/y^2}\mu(a)^2(N_{I_y}(a)^2 - N_{I_y}(a))
 & \ge
    \frac{y^2}{x}
      \bigg(
       \sum_{a \le x/y^2} \mu(a)^2 N_{I_y}(a)
      \bigg)^2
     - \sum_{a \le x/y^2} N_{I_y}(a)
 \\
  & \gg 
     \frac{x}{(\log x)^2}.
\end{align*}
Letting   
$
  \mathscr{I} 
 = \{[2^{j-1},2^{j}) : (\log x)^2 \le 2^{j} \le x^{1/6}\} 
$
and applying \eqref{eq:3.3}, we conclude that 
\[
  \SS(x)
   \gg
    \sum_{I_y \in \mathscr{I}} \frac{x}{(\log x)^2}
     \gg \frac{x}{\log x}.
\]
\end{proof}

\subsection{The upper bound}
 \label{subsec:3.2}

By \eqref{eq:3.1} we have 
$
 \SS(x + 1) = 2\SS_1(x) + 2\SS_2(x),
$
where 
\begin{align}
 \label{eq:3.6}
  \begin{split}
   \SS_1(x)
  & \defeq 
     \sum_{a \le x^{2/3}} 
      \mu(a)^2
       \sum_{n \le \sqrt{x/a}}
        \hspace{3pt}
         \sum_{m < n}
          \ind{\PP}(am^2 + 1) 
           \ind{\PP}(an^2 + 1)
 \\
 & \le
    \sum_{a \le x^{2/3}} 
     \mu(a)^2
      \bigg(
       \sum_{n \le \sqrt{x/a}}
        \ind{\PP}(an^2 + 1)
      \bigg)^2
 \end{split} 
 \intertext{and}
  \label{eq:3.7}
 \begin{split}
 \SS_2(x) 
  & \defeq
   \sum_{x^{2/3} < a \le x} \mu(a)^2
    \sum_{n \le \sqrt{x/a}}
     \hspace{3pt}
      \sum_{m < n}
       \ind{\PP}(am^2 + 1) 
        \ind{\PP}(an^2 + 1) 
 \\
 & \le 
    \sum_{n < x^{1/6}}
     \hspace{3pt}
      \sum_{m < n}
       \sum_{a \le x/n^2}
        \hspace{3pt}
         \ind{\PP}(am^2 + 1)
          \ind{\PP}(an^2 + 1).
 \end{split}
\end{align}

\begin{lemma}
 \label{lem:3.2}
\textup{(}i\textup{)}
Uniformly for $x \ge 2$ and $1 \le a \le x^{2/3}$, we have 
\[
  \sum_{n \le \sqrt{x/a}}
   \ind{\PP}(an^2 + 1)
    \ll
     \frac{\sqrt{x/a}}{\log x}
      \frac{a}{\phi(a)}
       \prod_{2 < p \le \sqrt{x}}
        \br{1 - \frac{(-a/p)}{p}}.
\]
\textup{(}ii\textup{)}
Uniformly for $1 \le m < n < x^{1/3}$, we have
\[
  \sum_{a \le x/n^2} 
   \ind{\PP}(am^2 + 1)\ind{\PP}(an^2 + 1)
    \ll
     \frac{x}{(n\log x)^2}
      \cdot
       \frac{(m,n)}
            {\phi((m,n))}
        \cdot 
         \frac{n^2 - m^2}{\phi(n^2 - m^2)}.
\]
\end{lemma}

\begin{proof}
(i) Given $x \ge 2$ and $1 \le a \le x^{2/3}$, let 
\[
 \rho_a(p)
  \defeq 
   \#\{b \bmod p : ab^2 + 1 \equiv 0 \bmod p\}.
\]
A routine application of Brun's sieve \cite[Theorem 2.2]{HR} gives    
\[
 \sum_{n \le \sqrt{x/a}}
  \ind{\PP}(an^2 + 1)
   \ll 
    \sqrt{x/a} 
     \prod_{p \le \sqrt{x}} 
      \br{1 - \frac{\rho_a(p)}{p}}.
\]
Since $1-\rho_a(p)/p=(1-1/p)(1-(\rho_a(p)-1)/(p-1))$, 
Mertens' theorem gives 
\[
 \prod_{p \le \sqrt{x}} 
  { \br{1 - \frac{\rho_a(p)}{p}} }
    \ll
     \frac{1}{\log x}
      \prod_{2<p \le \sqrt{x}}
       {\br{1 - \frac{\rho_a(p) - 1}{p - 1}} }.
\]
Now, $\rho_a(p)-1=(-a/p)$  for odd $p\nmid a$, and $\rho_a(p)=0$ 
for $p\mid a$, hence 
\[
 \prod_{2<p \le \sqrt{x}}
  { \br{1 - \frac{\rho_a(p) - 1}{p - 1}} }
   \le
    \frac{a}{\phi(a)}
     \prod_{2 < p \le \sqrt{x}}
     {\br{1 - \frac{(-a/p)}{p - 1}} },
\]
which proves the inequality in the lemma with $p - 1$ in the 
denominator instead of $p$.
But
$
 1 - (-a/p)/(p - 1)
  =
   \br{1 - (-a/p)/p}
    \br{1 + \Oh[1/p^2]} 
$
so the bound in the lemma holds.

(ii) Given $1 \le m < n < x^{1/3}$, let 
\[
 \rho_{m,n}(p)
  \defeq 
   \#\{b \bmod p : (bm^2 + 1)(bn^2 + 1) \equiv 0 \bmod p\}.
\]
Again by Brun's sieve \cite[Theorem 2.2]{HR}, 
\[
 \sum_{a \le x/n^2}
  \ind{\PP}(am^2 + 1)\ind{\PP}(an^2 + 1)
   \ll 
    \frac{x}{n^2}
     \prod_{p \le \sqrt{x}} 
      \br{1 - \frac{\rho_{m,n}(p)}{p}}.
\]
By Mertens' theorem we have
\begin{align*}
 \prod_{p \le \sqrt{x}}
  \br{1 - \frac{\rho_{m,n}(p)}{p}}
 & =
    \prod_{p \le \sqrt{x}}
     \br{1 + 
      \frac{p(2 - \rho_{m,n}(p)) - 1}{(p - 1)^2}}
        \br{\frac{p - 1}{p}}^2
 \\
 & \ll
    \frac{1}{(\log x)^2}
     \prod_{p \le \sqrt{x}}
      \br{1 
     + \frac{p(2 - \rho_{m,n}(p)) - 1}{(p - 1)^2}}.
\end{align*}

Now, for any prime $p$ we have  
\begin{align*}
  \rho_{m,n}(p) 
  =
   \begin{cases}
    2 & \text{if $p \nmid mn(m^2 - n^2)$,} \\
    1 & \text{if $p \mid mn(m^2 - n^2)$ and $p \nmid (m,n)$,} \\
    0 & \text{if $p \mid (m,n)$,}
   \end{cases}   
\end{align*}
hence 
\begin{align*}
 \prod_{p \le \sqrt{x}}
  \br{1 + \frac{p(2 - \rho_{m,n}(p)) - 1}{(p - 1)^2}}
 & \le
    \prod_{p \mid (m,n)}
     \br{\frac{p}{p - 1}}^2
      \prods[p \mid m^2 - n^2][p \nmid (m,n)]
       \frac{p}{p - 1}
 \\
 & =
    \prod_{p \mid (m,n)}
     \frac{p}{p - 1}
      \prod_{p \mid (m^2 - n^2)}
       \frac{p}{p - 1}. 
\end{align*}
Combining gives the result.
\end{proof}

\begin{proof}
[Deduction of the upper bound]
By \eqref{eq:3.6}, Lemma \ref{lem:3.2} (i) and Lemma 
\ref{lem:2.6}, we have
\[
 \SS_1(x)
  \ll
   \frac{x}{(\log x)^2}
    \sum_{a \le x^{2/3}} 
     \frac{a\mu(a)^2}{\phi(a)^2}
      \prod_{2 < p \le \sqrt{x}}
       \br{1 - \frac{(-a/p)}{p}}^2
        \ll 
         \frac{x}{\log x}.
\]
By \eqref{eq:3.7} and Lemma \ref{lem:3.2} (ii) we have
\[
 \SS_2(x) 
  \ll 
   \frac{x}{(\log x)^2}
    \sum_{n < x^{1/6}}
     \hspace{3pt}
      \frac{1}{n^2}
       \sum_{m < n}
        \hspace{3pt}
         \frac{(m,n)}{\phi((m,n))}
          \cdot 
           \frac{n^2 - m^2}{\phi(n^2 - m^2)}.
\]
To bound the double sum, we write $g = (m,n)$, $m = gm_1$, 
$n = gn_1$ and change the order of summation to obtain
\begin{align*}
  & \sum_{g \le x^{1/6}} \frac{1}{g^2}
     \hspace{3pt}
      \sum_{n_1 \le x^{1/6}/g}
       \frac{1}{n_1^2}
        \hspace{3pt}
         \sums[m_1 < n_1][(m_1,n_1) = 1] 
          \hspace{3pt}
           \frac{g}{\phi(g)}
            \cdot 
             \frac{g^2(n_1^2 - m_1^2)}{\phi(g^2(n_1^2 - m_1^2))}
 \\
  & \hspace{60pt} \le
    \sum_{g \le x^{1/6}} \frac{1}{\phi(g)^2}
     \hspace{3pt}
      \sum_{n_1 \le x^{1/6}/g}
       \frac{1}{n_1^2}
        \hspace{3pt}
         \sums[m_1 < n_1][(m_1,n_1) = 1] 
          \hspace{3pt}
           \frac{n_1^2 - m_1^2}{\phi(n_1^2 - m_1^2)}.
\end{align*}
This is equal to $\Oh[\sum_{n_1 \le x} 1/n_1] = \Oh[\log x]$
by Lemma \ref{lem:2.1} (ii).
Recalling that $\SS(x) = 2\SS_1(x) + 2\SS_2(x)$ and combining gives
\[
 \SS(x) \ll \SS_1(x) + \SS_2(x) \ll \frac{x}{\log x}.
\]
This completes the proof of the theorem.
\end{proof}


\end{document}